\documentclass[12pt]{amsart}

\usepackage{amsmath,amssymb,amsthm,amsfonts}
\usepackage{graphicx}
\usepackage{hyperref}
\usepackage{subfigure}

\title[Steiner symmetrization]
{Steiner symmetrization along a certain
equidistributed sequence of directions in the plane}
\author{Reza Asad and Almut Burchard}
\address{University of Toronto, Department of Mathematics,
40 St. George Street, Toronto, Canada M5S 2E4.
{\tt almut@math.toronto.edu}}
\date{September 7, 2015}

\theoremstyle{plain}
\newtheorem*{theo*}{Main Result}
\newtheorem{lemma}{Lemma}

\theoremstyle{definition}
\newtheorem{defn}{Definition}

\usepackage{caption}

\newcommand{\R}{{\mathbb R}}
\newcommand{\C}{{\mathbb C}}

\newcommand{\Z}{{\mathbb Z}}

\newcommand{\calJ}{{\mathcal{J}}}

\renewcommand{\Re}{{\rm Re}\,}
\renewcommand{\Im}{{\rm Im}\,}
\newcommand{\Cc}{\mathcal{C}_c^+}

\usepackage{dsfont}

\begin{document}

\begin{abstract}
This note reports the results of an undergraduate research
project from the year 2013-14, concerning 
the convergence of iterated Steiner symmetrizations in the plane.
The directions of symmetrization are chosen according to
the Van der Corput sequence, a classical
example of a sequence that is equidistributed in $S^1$
with low discrepancy. It is shown here that the resulting
iteration of Steiner symmetrizations
converges to the symmetric decreasing rearrangement.
The proof exploits the self-similarity of the sequence
of angular increments, using the technique of 
competing symmetries.
\end{abstract} \maketitle

Symmetrizations are rearrangements of 
functions that are used in geometry and analysis. 
The most important one is the
symmetric decreasing rearrangement, which replaces a given
nonnegative function by an equimeasurable radially decreasing
function. Another example is
Steiner symmetrization, which is a simpler rearrangement
that produces a reflection symmetry.

In typical applications, the goal is to reduce a rotationally symmetric 
optimization problem to 
the radial case, which can be treated as a single-variable problem. 
To this end, one establishes suitable inequalities between a function 
and its symmetrization. Functionals involving the gradient
(such as the Dirichlet energy) generally
decrease under symmetrization, while convolution functionals 
(such as the Coulomb energy) 
increase. In many cases, one can argue that optimizers 
are invariant under Steiner symmetrization in all directions, and therefore 
radial. This strategy has been used to solve extremal problems
for the perimeter, capacity, torsional rigidity, 
fundamental frequency, and related quantities in geometry and 
mathematical physics~(see \cite{PSz1951}). Symmetrization techniques 
have been also used to determine sharp constants for
the Young, Sobolev, and Hardy-Littlewood-Sobolev 
inequalities of functional analysis~(see~\cite[Chapters 3 and 4] {LL2001}).

It is often useful to approximate the symmetric decreasing rearrangement
by concatenating Steiner symmetrizations along a suitable sequence
of directions.  For a random sequence of directions,
this almost surely converges to the symmetric decreasing 
rearrangement. 
What determines whether a given sequence of directions
produces convergence? 

This question has received some attention in the literature.
Convergence has been established for a number of explicit examples. 
On the other hand, 
there are sequences of directions that are dense,
and even equidistributed in the sense of Weyl, 
such that the corresponding sequence of
Steiner symmetrizations fails to converge. For a 
full discussion of the
state-of-the-art we refer to a recent paper of 
Bianchi {\em et al.}~\cite{BBGV}. The authors raise the 
question whether every sequence of directions that is
equidistributed and has low discrepancy 
gives rise to a sequence of Steiner
symmetrizations that converges to the symmetric decreasing 
rearrangement.

In this note, we consider the special case of two dimensions, 
and show that the {\em van der Corput sequence}, a classical example
of a sequence of low discrepancy in $S^1$, indeed produces convergence. 
Our proof takes advantage of the self-similar construction of this
sequence. 
The question remains open for general sequences of low discrepancy.

\smallskip 
We start with some definitions.
Let  $\Cc$ denote the set of nonnegative
continuous functions with compact support
in the complex plane. We work with the topology of uniform
convergence, defined by the distance function
$$
||f-g||=\sup_{z\in\C} |f(z)-g(z)|\,.
$$

\begin{defn} The {\em symmetric decreasing rearrangement} 
$f^*$ of a function $f \in \Cc$ is
the unique symmetric decreasing function that is equimeasurable 
with $f$, i.e.,
the level set $\{ z\in \mathbb{C} \vert f^*(z)>t\}$ has the same measure as the corresponding level
set of $f$ for each $t>0$ (see Figure~\ref{fig:f*}).
\end{defn}
\begin{figure}[h!]
\centering
\includegraphics[scale = 1.6]{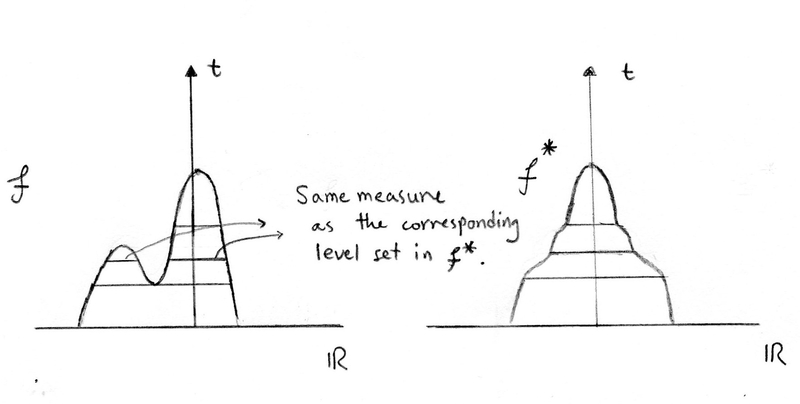}
  \caption{\small The symmetric decreasing rearrangement 
of a function $f\in\Cc$}
  \label{fig:f*}
\end{figure}

The function $f^*$ is again a nonnegative continuous function with 
compact support. In fact, the modulus of continuity of $f^*$ is bounded 
by the modulus of continuity of $f$, and its support is contained 
in the smallest centered circle that contains the support of $f$.
The symmetric decreasing rearrangement is {\em non-expansive},
i.e., 
$$
||f^*-g^*|| \le ||f-g||\qquad (f,g\in\Cc)\,,
$$
see~\cite[Ex. 1.7, Ex. 2.14]{B2009} 
and~\cite[Section 3.4]{LL2001}.

\begin{defn} \label{def:Su}
The {\em Steiner symmetrization}  of  a function $f\in\Cc$ is the function $Sf$ with the 
property that its restriction to each line $\Re z = t$ is the 
unique symmetric decreasing function of $\Im z$
that is equimeasurable with the corresponding restriction of $f$ 
(see Figure~\ref{fig:steiner}).
For  any angle $\alpha \in \R/(2\pi)$, the 
{\em Steiner symmetrization in direction $\alpha$} 
is given by $$
S_\alpha:=R_\alpha SR_{-\alpha}\,,
$$
where the rotation $R_\alpha$ acts on 
$\Cc$ by $R_\alpha f(z)=f(e^{-i\alpha}z)$ for $z\in\C$.
\end{defn}

\begin{figure}[h!]
\centering
\includegraphics[scale = 0.7]{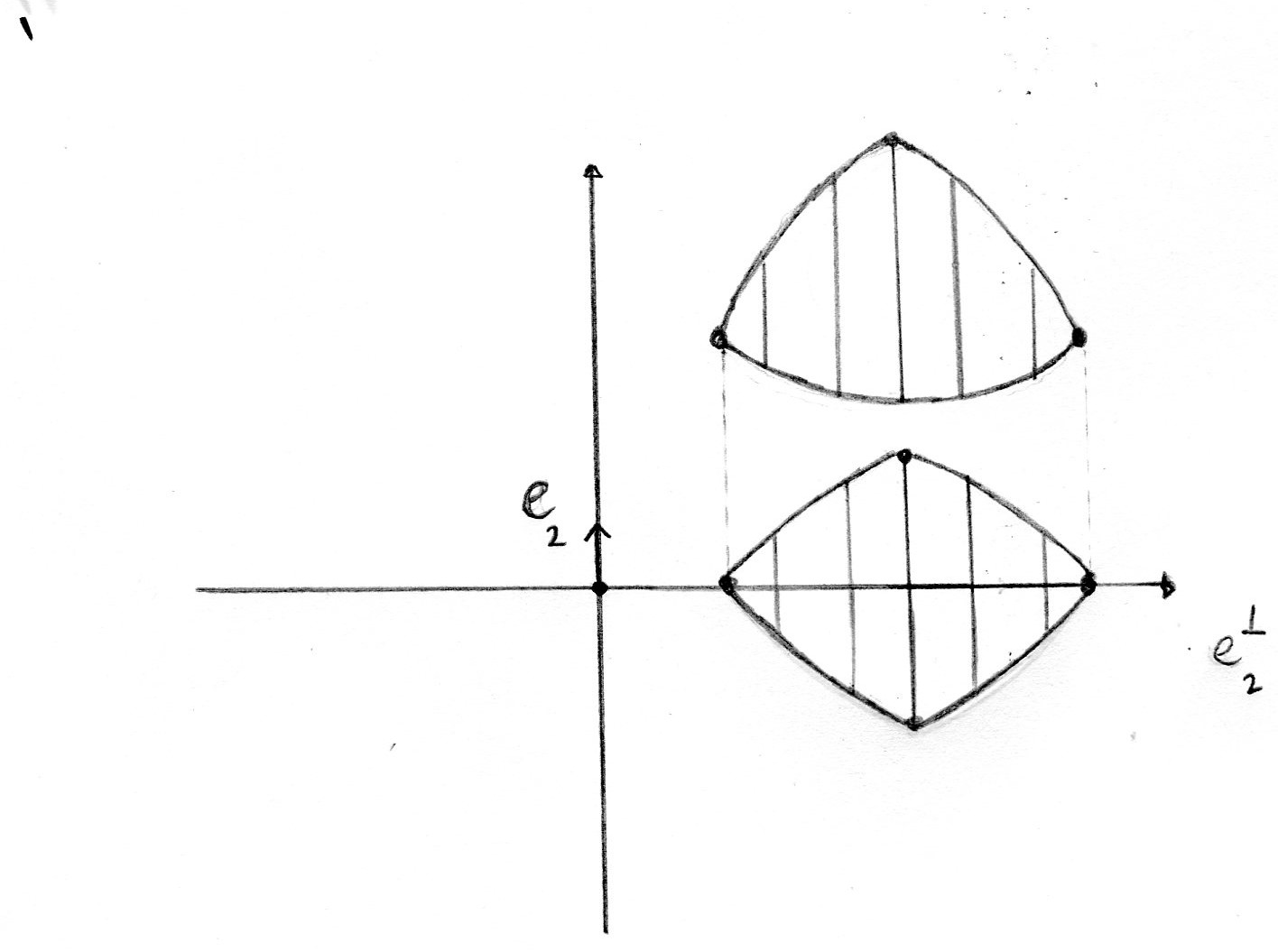}
  \caption{\small A level set of $f$ and the corresponding level set of $Sf$
}
  \label{fig:steiner}
\end{figure}
Similar to the symmetric decreasing rearrangement,
Steiner symmetrization defines a transformation on 
$\Cc$ that is non-expansive, with 
$$
\|Sf - Sg\| \le \|f-g\|\qquad (f,g\in\Cc)\,.
$$
A useful consequence is that $S_\alpha$ depends continuously on 
the direction~$\alpha$. We next choose the sequence
of directions.

\begin{defn}
The sequence  $\left(e^{i\theta_n}\right)_{n\ge 0}$ with arguments
$$
\theta_n=\pi \sum_{j=0}^k a_j2^{-j}
\qquad \mbox{for}\ n = \sum_{j=0}^k a_j2^j
$$
is called the {\em van der Corput sequence} on $S^1$  (see Figure~\ref{fig:VdC}).
\end{defn}

\begin{figure}[htbp]
    \centering
    \includegraphics[width = 0.3\linewidth]{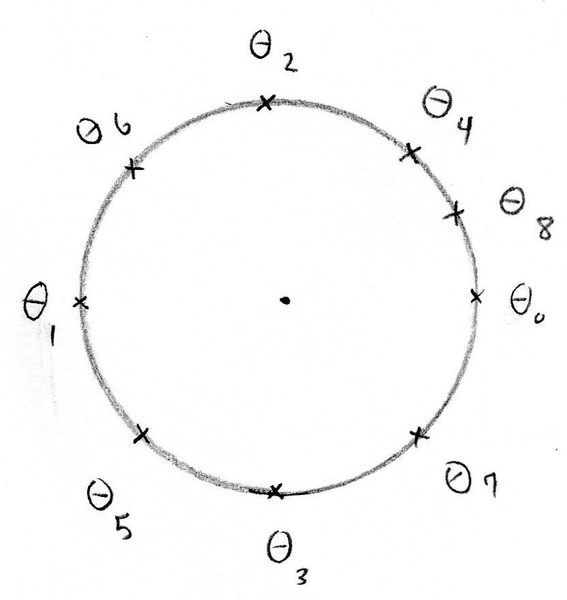}
    \caption{\small The Van der Corput Sequence on the unit circle.}
 \label{fig:VdC}
\end{figure}

In other words, if $n$ has binary expansion $(a_k\dots a_0)$,
then $\theta_n/(2\pi)$ has
binary expansion $(0.a_0 \dots a_k)$.
The van der Corput sequence has the equidistribution property that
the fraction of points $(e^{i\theta_n})_{n\le N}$ 
that fall into any given arc $A$ converges to its length 
$\ell(A)$ as a fraction of $2\pi$ when $N\to\infty$. The 
{\em discrepancy}
$$ 
D_N = \sup_{A \in S^1} \left|\frac {\mbox{\rm \#} 
\{ n<N:e^{i\theta_n} \in A\}}{N} - \frac {\ell(A)}{2\pi}\right|, 
$$
measures how quickly the sequence $(e^{i\theta_n})$ converges to
the uniform distribution on $S^1$. According to a theorem
of Roth, the discrepancy of every sequence
must exceed a constant multiple of $N^{-1}\log N$
for infinitely many $N$. The discrepancy of the
van der Corput sequence is bounded from above
and below by constant multiples of the optimal order 
$N^{-1}\log N$ (for all $N$). Sequences with this property
are said to have {\em low discrepancy}. For more information about 
equidistributed sequences, we refer to~\cite{KN1974}.

\smallskip

\begin{theo*}\label{theo:main}
Let $(e^{i\theta_n})$ be the van der Corput sequence on the 
unit circle, let $f\in \Cc$, and let $f^*$ be its symmetric
decreasing rearrangement. Then the sequence $(f_n)$
defined by
\begin{equation} \label{eq:def-fn}
f_0=S_{\theta_0}f\,,\qquad 
f_n := S_{\theta_n}f_{n-1} \quad (n\ge 1)
\end{equation}
converges uniformly to $f^*$.
\end{theo*}

The proof requires some auxiliary results.
To keep track how much a function $g\in\Cc$
differs from its symmetric decreasing rearrangement,
we use the functional
\begin{equation} \label{eq:def-J}
\calJ(g)=\int_{\C} g(z)\, e^{-|z|^2}\, d^2z\,.
\end{equation}
Clearly, $\calJ$ is continuous on $\Cc$.  
The first lemma shows that Steiner symmetrizations and rotations
play the roles of {\em competing symmetries}
for this functional~\cite{CL1990}.

\begin{lemma}\label{lem:HL}
Let $g \in \Cc$ and $\alpha\in \R/(2\pi)$. Then
$$
\calJ(R_\alpha g)= \calJ(g)\quad\mbox{and}\quad
\calJ(S_\alpha g)\ge \calJ(g)\,.
$$ 
The inequality is strict unless $S_\alpha g=g$. 
\end{lemma}

\begin{proof} 
The rotational invariance of $\calJ$ follows directly
from the fact that the Gaussian $e^{-|z|^2}$
is radial. For the second claim, we use Fubini's theorem to write
$$
\calJ(S g) = \int_\R \left(\int_\R Sg(x+iy)\,e^{-y^2}\, dy
\right)e^{-x^2}dx\,.
$$
According to
Definition~3, the vertical cross section of $Sg$ 
at any given $x$ is the one-dimensional symmetric decreasing
rearrangement of the corresponding cross section of $g$.
By~\cite[Theorem 3.4]{LL2001}, the inner integral satisfies
$$
\int_\R Sg(x+iy)\,e^{-y^2}\, dy \ge \int_\R g(x+iy)\, e^{-y^2}\, dy\,,
$$
with equality if and only if the integrands agree almost 
everywhere. Since $g$ is continuous, this means that
$\calJ(S g) \ge \calJ(g)$, with equality if and only 
if $Sg=g$. The claim for $S_\alpha$ then follows from the
rotational invariance of~$\calJ$.  
\end{proof}

\smallskip 
The next lemma says that the consecutive differences in the
van der Corput sequence repeat themselves infinitely often. 
The convergence of similarly repetitive sequences of 
symmetrizations was analyzed 
by Van Schaftingen~\cite{vS2006u} and by Klain~\cite{K2012}.

\begin{lemma} \label{lem:selfsimilar}
Let $(e^{i\theta_n} )$ be the van der Corput sequence on 
the unit circle. The differences $\gamma_n = \theta_n - \theta_{n-1}$ 
satisfy
$$
\gamma_{2^{j}+n}  = \gamma_n\qquad (1\le n < 2^j)
$$
for all $j\ge 1$. Moreover, $\displaystyle{\lim_{j\to\infty} \gamma_{2^j}=0}$
modulo $2\pi$.
\end{lemma}

\begin{proof}
We expand
$n=\sum_{i<j} a_i2^{-i}$ and $n-1=\sum_{i<j} b_i2^{-i}$ to obtain
$$
\theta_{2^{j}+n} = \pi\left(2^{-j} + \sum_{i=0}^{j-1} a_i2^{-i}\right)\,,\quad  
\theta_{2^{j}+n-1} = \pi\left(2^{-j} + \sum_{i=0}^{j-1} b_i2^{-i}\right)\,.
$$
The difference is given by
\begin{align*}
\theta_{2^{j}+n} - \theta_{2^{j}+n-1}
&=\pi\left( \sum_{i=0}^{j-1} a_i2^{-i} - \sum_{i=0}^{j-1} b_i2^{-i}\right)\\
& = \theta_n - \theta_{n-1}\,,
\end{align*}
proving the first claim.
We similarly compute
\begin{align*}
\theta_{2^j} - \theta_{2^j - 1}
&= \pi\left(2^{-j} - \sum_{i = 0}^{j-1} 2^{-i} \right)
\\
&= -2\pi + \frac{3\pi} {2^j},
\end{align*}
proving the second claim.
\end{proof}

The final lemma gives  a  characterization of radially decreasing functions.

\begin{lemma}
\label{lem:identify} 
Let $h\in\Cc$. If $Sh=h$, and $S_{\alpha_j}h=h$ for some sequence
$(\alpha_j)$ of non-zero angles in $\R/(2\pi)$
that converges to zero, then $h=h^*$.
\end{lemma}

\begin{proof} By assumption, $h$ is symmetric
under reflection at the real axis and at
each of the lines $z=te^{i\alpha_j}$. 
Since the composition
of a pair of reflections equals the rotation by twice the enclosed angle,
$h$ is symmetric under each of the rotations $R_{2\alpha_j}$. 
Expanding $h$ and $R_{2\alpha_j}h$
in polar coordinates as a Fourier series,
$$
h(re^{i\phi})= \sum_{m\in\Z} a_m(r) e^{im\phi}\,,
\qquad 
R_{2\alpha_n}h(re^{i\phi})
= \sum_{m\in\Z} a_m(r)e^{im(\phi-2\alpha_j) }\,,
$$
we see upon comparing coefficients that
$$
a_m(r)\bigl(1-e^{-2im\alpha_j}\bigr) =0
$$
for all integers $m$, all $r\ge 0$, and every index $j$.
Since each $\alpha_j\ne 0$ but $\lim \alpha_j=0$,
for any given $m$ we can choose $j$ so large that
the exponent $2m\alpha_j$ is not an integer multiple of $2\pi$,
and consequently the factor $(1-e^{-2im\alpha_j})$ does not vanish.
Therefore $a_m(r)$ must vanish identically for each $m\ne 0$,
and $h$ is radial. Finally, since $Sh=h$, 
the restriction of $h$ to the imaginary axis is symmetric decreasing, 
and we conclude that $h=h^*$. 
\end{proof}

\smallskip We note in passing that the conclusion of Lemma~\ref{lem:identify} remains valid, if $(\alpha_n)$ is {\em any} sequence
(not necessarily converging to zero) that assumes infinitely
many distinct values in $\R/(2\pi)$. The reason is that
the corresponding reflections generates a
dense subgroup of $O(2)$~\cite{L1967}.

\smallskip 
\begin{proof} [Proof of the main result.]
Let  $(e^{i\theta_n})$ be the van der Corput sequence on 
$S^1$, and let $(f_n)$ be the sequence constructed from 
$f$ by iterated Steiner symmetrization, as in
Eq.~(\ref{eq:def-fn}). Then $g_n=R_{-\theta_n} f_n$ satisfies
the recursion relation
\begin{equation}\label{eq:recursion}
g_n= S R_{-\gamma_n}g_{n-1}\,\qquad (n\ge 1)\,,
\end{equation} 
where $\gamma_n = \theta_n- \theta_{n-1}$.
Since Steiner symmetrization improves the 
modulus of continuity, the sequence $(g_n)$ is equicontinuous,
uniformly bounded, and supported on a common ball.
By the Arzel\`a-Ascoli theorem, there exists a subsequence 
$(g_{n_k})$ that converges uniformly to some function $h$. 
We want to show that the entire sequence converges, and
that $h=f^*$.

Let $\calJ$ be the functional defined in Eq.~\eqref{eq:def-J}.
By continuity, $\calJ(g_{n_k})$ converges to $\calJ(h)$. 
By Lemma~\ref{lem:HL}, the value of $\calJ$ 
increases monotonically along $(g_n)$, hence 
$$\lim_{n\to\infty} \calJ(g_n) =\sup_n\calJ(g_n)\\
=\calJ(h)
$$
along the entire sequence.

By construction, $Sh=h$. Fix $n\ge 1$.
By Lemma~\ref{lem:selfsimilar} and Eq.~\eqref{eq:recursion},
\begin{align*}
SR_{-\gamma_n}g_{2^j+n-1} &= SR_{-\gamma_{2^j+n}}g_{2^j+n-1}\\
&= g_{2^j+n}
\end{align*}
for every $j$ with $2^j>n$;
in particular, $\calJ(SR_{-\gamma_n}g_{2^j+n-1}) = \calJ(g_{2^j+n})$. Taking $j\to\infty$ yields that
$$\calJ(SR_{-\gamma_n}h) = \calJ(h)\,.
$$
It follows from the rotational invariance of $\calJ$
that $\calJ(S_{\gamma_n}h)= \calJ(h)$,
and from the second part of Lemma~\ref{lem:HL} that 
$$
S_{\gamma_n}h=h \qquad (n\ge 1)\,.
$$
We use this equation with $n=2^j$. 
By Lemma~\ref{lem:selfsimilar},
$(\gamma_{2^j})$ converges to zero in $\R/(2\pi)$, and by 
Lemma~\ref{lem:identify}, $h$ is radially decreasing.
Let $(g_{n_k})$ be the convergent subsequence chosen in the
first part of the proof.
Since $h^*=h$ and $(g_{n_k})^*=f^*$ for each~$n$, 
the non-expansive property of rearrangements implies that 
$$
 ||h-f^*||\le ||h-g_{n_k}||
$$
for each $k$. We finally take $k\to\infty$ and conclude that
$h=f^*$. By monotonicity, $||f^*-g_n||$ converges to
zero along the entire sequence, and by
the rotational invariance of the norm
$||f^*-f_n||$ converges to zero as well. 
\end{proof}

\section*{Acknowledgment} This work was supported in 
part by NSERC though a USRA summer project (R.A.) and
a Discovery grant (A.B.)

\bibliographystyle{plain}

\end{document}